\documentclass{amsart}
\usepackage[latin1]{inputenc}				% Pour les caractères accentués
\usepackage[T1]{fontenc}					% Encodage de caractères
\usepackage{lmodern}						% Police vectorielle Latin Modern
\usepackage[english]{babel}					% Règles typographiques françaises
\usepackage{amsmath,amsfonts}				% Symboles mathématiques
\usepackage{amssymb}						% Symboles mathématiques bis
\usepackage{graphicx}						% Insérer des images
\usepackage{dsfont}							% Pour le symbole de l'indicatrice
\usepackage{pst-all}						% Pour les graphiques
\usepackage{stmaryrd}                       % Pour les intervalles d'entiers
\usepackage{multirow}						% Pour les colonnes multiples
\usepackage{amsaddr}						%Pour les coordonnees

\usepackage{bm}								% Pour les caracteres gras en mode mathematiques
\usepackage{amsthm}

\usepackage{a4wide} 						%Pour avoir des marges standard

\newtheorem{thm}{Theorem}[section]
\newtheorem*{nota}{Notation}
\newtheorem{rem}[thm]{Remark}
\newtheorem{lem}[thm]{Lemma}
\newtheorem{prop}[thm]{Proposition}

\newtheorem{coro}[thm]{Corollary}

\newcommand{\R}{\mathbb{R}}

\newcommand{\N}{\mathbb{N}}
\newcommand{\Z}{\mathbb{Z}}

\newcommand{\pr}{\mathcal{P}}

\title[An Extension of a Theorem of Duffin and Schaeffer ]{An Extension of a Theorem of Duffin and Schaeffer in Diophantine Approximation}
\author{Faustin ADICEAM}
\date{}

\makeatletter
\def\imod#1{\allowbreak\mkern10mu({\operator@font mod}\,\,#1)}
\makeatother

\address{Department of Mathematics, National University of Ireland at Maynooth}
\email{fadiceam@gmail.com}

\begin{document}

\begin{abstract}
Duffin and Schaeffer have generalized the classical theorem of Khintchine in metric Diophantine approximation in the case of any error function under the assumption that all the rational approximants are irreducible. This result is extended to the case where the numerators and the denominators of the rational approximants are related by a congruential constraint stronger than coprimality.
\end{abstract}

\maketitle

For convenience, the following notation shall be used throughout~:
\paragraph{\textbf{\large{Notation}}\\}
\begin{itemize}
\item $ \lfloor x \rfloor$ ($x \in \R$)~: the integer part of $x$.
\item $ \llbracket x , y \rrbracket$ ($x, y\in\R$, $x \le y$)~: interval of integers, i.e. $\llbracket x , y \rrbracket = \left\{ n\in\Z \; : \; x\le n \le y\right\}$. 
\item $\textrm{Card}(X)$ or $|X|$~: the cardinality of a finite set $X$.
\item $A^{\times}$~: the set of invertible elements of a ring $A$.
\item $\pr$~: the set of prime numbers.
\item $\pi$~: any prime number. 
\item $\varphi (n)$~: Euler's totient function.
\item $\tau (n)$~: the number of divisors of a positive integer $n$.
\item $\omega (n)$~: the number of distinct prime factors dividing an integer $n\ge 2$ ($\omega (1)=0$).
\end{itemize}

\section{Introduction and statement of the result}

\paragraph{}
The well--known theorem of Duffin and Schaeffer~\cite{dufsch} in metric number theory extends the classical theorem of Khintchine in the following way~:

\begin{thm}[Duffin \& Schaeffer, 1941]\label{dufsch} 
Let $\left(q_k \right)_{k\ge 1}$ be a strictly increasing sequence of positive integers and let $(\alpha_k)_{k\ge 1}$ be a sequence of non--negative real numbers which satisfies the conditions~:
\begin{align*}
& \textbf{(a)} \quad \sum_{k=1}^{+\infty} \alpha_{k} = +\infty , \\
& \textbf{(b)} \quad \sum_{k=1}^n\frac{\alpha_k \, \varphi\left(q_k\right)}{q_k} > c\sum_{k=1}^n \alpha_k \textrm{  for arbitrarily many integers } n\ge 1 \textrm{ and a real number } c>0.
\end{align*}

Then for almost all $x\in\R$ there exist arbitrarily many relatively prime integers $p_k$ and $q_k$ such that  $$\left|x-\frac{p_k}{q_k} \right| < \frac{\alpha_k}{q_k}\cdotp$$
\end{thm}

Here as elsewhere, \textit{almost all} must be understood in the sense that the set of exceptions has Lebesgue measure zero.

\paragraph{}
Several generalizations of Theorem~\ref{dufsch} have been considered~: on the one hand, the conjecture of Duffin and Schaeffer asks whether assumption~(b) may be weakened in the statement of the result by replacing it by the divergence of the series $\sum_{k=1}^n\alpha_k \, \varphi\left(q_k\right)q_k^{-1}$. Even if the analogue of this issue has been proved in higher dimensions~\cite{dufschdim} or with some extra assumptions on the sequence $\left(\alpha_k \right)_{k\ge 1}$~\cite{dufschextra}, the full conjecture is still open. On the other hand, one may try to see to what extent Theorem~\ref{dufsch} remains true when the numerators $p_k$ and the denominators $q_k$ of the fractional approximations are related by some stronger relationship (in a sense to be made precise) than coprimality.

Indeed, metric Diophantine approximation results in one dimension when the denominators of the rational approximants are confined to a prescribed set are numerous (see for instance~\cite{bugeaud}, Theorem~5.9). However,  restrictions on numerators introduce new difficulties which do not always seem to be easy to overcome (see~\cite{bugeaud}, p.114 for an account on this fact). In a series of articles, \cite{g1}, \cite{g2}, \cite{g3} \& \cite{g4}, G.Harman tackled the problem and gave several results in the case where denominators and numerators were confined to vary within independent sets of integers. The main theorem proved in this paper gives another approach to this problem studying the case where numerators and denominators are confined to dependent sets of integers in the sense that they are related, not only by the relation of Diophantine approximation of a given real number, but also by some congruential constraints~:

\begin{thm}(Extension of the theorem of Duffin and Schaeffer).\label{extensiondf}
Let $(q_k)_{k\ge 1}$ be a strictly increasing sequence of positive integers and let $(\alpha_k)_{k\ge 1}$ be a sequence of positive real numbers. Let $(a_k)_{k\ge 1}$ be a sequence such that for all $k\ge 1$, $a_k \in \left(\Z/\!\raisebox{-.65ex}{\ensuremath{q_k\Z}}\right)^{\times}$. For $k\ge 1$, denote by $G_k$ a subgroup of $\left(\Z/\!\raisebox{-.65ex}{\ensuremath{q_k\Z}}\right)^{\times}$ and by $a_kG_k$ the coset of $a_k$ in the quotient of $\left(\Z/\!\raisebox{-.65ex}{\ensuremath{q_k\Z}}\right)^{\times}$ by $G_k$. Assume furthermore that~:
\begin{align*}
& \textbf{(a)} \quad \sum_{k=1}^{+\infty} \alpha_{k} = +\infty , \\
& \textbf{(b)} \quad \sum_{k=1}^{n} \alpha_k \frac{\left| G_k \right|}{q_k} > c\sum_{k=1}^{n} \alpha_k \;\; \textrm{ for infinitely many positive integers } n \textrm{ and a real number } c>0, \\
& \textbf{(c)} \quad \frac{\varphi(q_k)}{q_k^{1/2-\epsilon} \left|G_k \right|} \longrightarrow 0 \; \textrm{ as } k \textrm{ tends to infinity,  for some } \epsilon>0.
\end{align*}

Then, for almost all $x \in\R$, there exist arbitrarily many relatively prime integers $p_k$ and $q_k$ such that 
\begin{align}\label{anecd}
\left|x - \frac{p_k}{q_k} \right| < \frac{\alpha_k}{q_k}\quad \textrm{and} \quad p_k\in a_kG_k.
\end{align}
\end{thm}  

\begin{rem}
In Theorem~\ref{extensiondf}, condition~(b) is obviously implied by the fact that 
\begin{align}\label{condc}
\frac{\left| G_k \right|}{q_k} > c >0
\end{align} 
for all $k \ge 1$ and for a real number $c>0$. However, if, instead of~(\ref{condc}), one can prove the weaker assertion  
\begin{align}\label{condcabel}
\sum_{k=1}^{n}\frac{\left| G_k \right|}{q_k} > c n
\end{align} 
for some $c>0$ and all integers $n\ge 1$, then, assuming that the sequence $(\alpha_k)_{k\ge 1}$ is non-increasing, condition~(b) still holds true. This may be seen by making an Abel transformation in the left-hand side of (b).

It is likely that formula~(\ref{condcabel}) can be proved for many sequences $\left(q_k\right)_{k\ge 1}$ that do not satisfy~(\ref{condc}). 
\end{rem}

As an application of Theorem~\ref{extensiondf}, consider a subsequence $\left(q_k^d\right)_{k\ge 1}$ of the $d^{th}$ powers of the natural numbers ($d\ge 1$ is an integer) and take for $G_k$ ($k\ge 1$) the group of $d^{th}$ powers in a reduced system of residues modulo $q_k$. For any $q\in\N$ denote furthermore by $r_d(q)$ the cardinality of the set of $d^{th}$ powers in a reduced system of residues modulo $q$ and set for simplicity
\begin{align}\label{defsdq}
s_d(q) :=\frac{r_d(q)}{q}\cdotp
\end{align}

\begin{coro}\label{thminteret}
Let $(q_k)_{k\ge 1}$ be a strictly increasing sequence of positive integers and let $(\alpha_k)_{k\ge 1}$ be a sequence of positive real numbers. Fix an integer $a\ge 1$ and assume furthermore that~:
\begin{align*}
& \textbf{(a)} \quad \sum_{k=1}^{+\infty} \alpha_{k} = +\infty , \\
& \textbf{(b)} \quad \sum_{k=1}^{n} \alpha_k s_d\left( q_k^d \right) > c\sum_{k=1}^{n} \alpha_k \;\; \textrm{ for infinitely many positive integers } n \textrm{ and a real number } c>0, \\
& \textbf{(c)} \quad \gcd(q_k , a)=1 \;\; \textrm{ for all } k\ge 1. 
\end{align*}
Then for almost all $x \in\R$, there exist arbitrarily many relatively prime integers $p_k$ and $q_k$ such that 
\begin{align*}
\left|x - \frac{p_k}{q_k^d} \right| < \frac{\alpha_k}{q_k^d}\quad \textrm{and} \quad p_k\equiv a b_k^d\imod{q_k}\; \textrm{ for some } b_k\in\Z \textrm{ relatively prime to } q_k.
\end{align*}
\end{coro}

Corollary~\ref{thminteret} answers a question which appeared in a problem of simultaneous Diophantine approximation of dependent quantities~: given an integer polynomial $P(X)$ and a real number $x$, what is the Hausdorff dimension of the set of real numbers $t$ such that $t$ and $P(t)+x$ are simultaneously $\tau$--well approximable, where $\tau>0$? The author proved~\cite{aparaitre} that such a simultaneous approximation implied an approximation of $x$ by a rational number $p/q^d$, where $d$ is the degree of $P(X)$ and where the integer $p$ satisfies the congruential constraint mentioned in the conclusion of Corollary~\ref{thminteret}, with $a$ the leading coefficient of $P(X)$. The emptiness of the set under consideration is obtained for almost all $x$ as a consequence of the convergent part of the Borel--Cantelli Lemma when $\tau > d+1$ and Corollary~\ref{thminteret} enables one to prove the optimality of this lower bound.

\paragraph{}
The paper is organized as follows~: first some lemmas of an arithmetical nature shall be recalled (section~\ref{sec2}). They shall be needed to prove Corollary~\ref{thminteret} in section~\ref{sec3}, where the modifications to make in the proof to prove Theorem~\ref{extensiondf} shall also be indicated.

\section{Some auxiliary results}\label{sec2}

In this section are collected various results which shall be needed later.

\subsection{Some lemmas in arithmetic}
For any integer $n\ge 2$, let $\tau (n)$ be the number of divisors of $n$ and let $\omega (n)$ be the number of \textit{distinct} prime factors dividing $n$. If $$n= \prod_{i=1}^r \pi_{i}^{\alpha_i}$$ is the prime factor decomposition of the integer $n$, recall that
\begin{align*}
\omega(n) & = r \quad \mbox{ and } \quad \tau(n) = \prod_{i=1}^r \left( \alpha_i +1 \right).
\end{align*}

The following lemma, which deals with some comparative growth properties about these two arithmetical functions, is well--known.

\begin{lem}\label{tau}
\begin{itemize}
\item For any $\epsilon >0$, $\tau(n)=o\left(n^ \epsilon \right)$.
\item For any $\epsilon >0$ and any positive integer $m$, $\omega(n)=o\left( \log n \right)$ and $m^{\omega(n)}=o\left(n^\epsilon \right).$
\end{itemize}
\end{lem}

\begin{proof}
See for instance~\cite{hw}, \S 22.11 and \S 22.13.
\end{proof}

If $n\ge 2$ and $d\ge 1$ are integers, recall that $r_d(n)$ denote the number of distinct $d^{th}$ powers in the reduced system of residues modulo $n$ and denote by $u_d(n)$ the number of $d^{th}$ roots of unity modulo $n$, that is,
\begin{align*}
r_d(n) &= \textrm{Card} \left\{ m^d \imod{n} \; : \; m \in \left(\Z/\!\raisebox{-.65ex}{\ensuremath{n\Z}} 
\right)^{\times} \right\}  \\
u_d(n)& = \textrm{Card} \left\{ m\in \Z/\!\raisebox{-.65ex}{\ensuremath{n\Z}} \; : \; m^d \equiv 1 \imod{n} \right\} .
\end{align*}
Set furthermore $r_d(1)=u_d(1)=1$. 

\begin{rem}\label{multirdq}
Let $u(f,n)$ be the number of solutions in $x$ of the congruence $$f(x):=\sum_{k=0}^{d} a_k x^k \equiv 0 \imod{n}$$ for a given polynomial $f\in \Z [X]$ of degree $d$. It is well--known that, as a consequence of the Chinese Remainder Theorem, $u(f,n)$ is a multiplicative function of $n$. It follows that $u_d(n)$ is multiplicative with respect to $n$ for any fixed $d$.
\end{rem}

The following proposition gives explicit formulae for $r_d(n)$ and $u_d(n)$. 

\begin{prop}\label{decompte}
The arithmetical functions $r_d(n)$ and $u_d(n)$ are multiplicative when $d$ is fixed.
Furthermore, if $n=\pi^k$, where $\pi \in \pr$ and $k\ge 1$ is an integer, then the following equalities hold~: $$r_d(n)=\frac{\varphi(\pi^k)}{u_d(\pi^k)} \quad and \quad u_d(n)=\left\{
    \begin{array}{ll}
        \gcd(2d, \varphi(n)) & \mbox{if } 2|d,\; \pi=2 \mbox{ and } k\ge 3 ,\\
        \gcd(d, \varphi(n)) & \mbox{otherwise,}
    \end{array}
\right.$$ 
where $\varphi$ is Euler's totient function.
\end{prop}

\begin{proof}
See \S 2 in~\cite{koro}.
\end{proof}

\subsection{Dirichlet characters and the P\'olya--Vinogradov inequality}\label{charact}

Let $G$ be a finite abelian group, written multiplicatively and with identity $e$. A \textit{character} $\chi$ over $G$ is a multiplicative homomorphism from $G$ into the multiplicative group of complex numbers. The image of $\chi$ is contained in the group of $|G|^{th}$ roots of unity.

It is readily seen that the set of characters over $G$ form a group, called the dual group of $G$ and written $\hat{G}$. Its unit $\chi_0$ is the \textit{principal} (or \textit{trivial}) \textit{character}, which maps everything in $G$ to unity.

The following is well--known (see~\cite{mendes}, chapter 7)~:

\begin{thm}\label{basecaract}
\begin{itemize}
\item[i)] There are exactly $|G|$ characters over $G$.
\item[ii)] For any $g\neq e$, $$\sum_{\chi \in \hat{G}} \chi(g) =0.$$
\item[iii)] For any non-principal character $\chi$, $$\sum_{g \in G} \chi(g) =0.$$
\end{itemize}
\end{thm}

If $n>1$ is an integer, consider the group $G = \left(\Z/\!\raisebox{-.65ex}{\ensuremath{n\Z}} \right)^{\times}$. A character $\chi$ over $G$ may be extended to all integers by setting $\chi(m) = \chi(m \imod n)$ if $\gcd(n, m) = 1$ and $\chi(m) = 0$ if $\gcd(n, m) > 1$. Such a function is called a \textit{Dirichlet character to the modulus}
$n$ and shall still be denoted by $\chi$.

In what follows, an upper bound on the sum of such characters over large intervals shall be needed. A fundamental improvement on the trivial estimate given by the triangle inequality is the P\'olya--Vinogradov inequality (see~\cite{mendes}, chapter 9)~:

\begin{thm}[P\'olya \& Vinogradov, 1918]\label{povi}
For any non principal Dirichlet characters $\chi$ over $\left(\Z/\!\raisebox{-.65ex}{\ensuremath{n\Z}} \right)^{\times}$ ($n>1$) and any integer $h$, the following holds~: $$\left|\sum_{k=1}^{h} \chi(k) \right| \le 2\sqrt{n}\log n.$$
\end{thm}

\begin{rem}
When $\chi$ is a so--called primitive character (which is the case if $n$ is prime), the multiplicative constant 2 in the above may be replaced by 1. This refinement shall not be needed.
\end{rem}

\section{The proof of the main result}\label{sec3}

The first part of this section shall be devoted to the proof of Corollary~\ref{thminteret}~: all the tools introduced in the previous section shall be used there. In the second subsection, all the modifications needed to prove Theorem~\ref{extensiondf} are given. 

\subsection{The proof of Corollary~\ref{thminteret}}

The proof of Corollary~\ref{thminteret} is a generalization of the proof of the theorem of Duffin and Schaeffer~\cite{dufsch}. All the new notation to be used is summarized in Figure~\ref{table1}.

\paragraph{}
\begin{figure}[!h]
\centering
\begin{tabular}{||c||c||c||}
\hline 
\hline
Notation & Parameters & Definition \\
\hline
\hline
$\varphi_{\mu}(n)$ & $n\ge 2$, $\mu>0$ & $\mathrm{Card}\left\{ l\in \llbracket 1 , \mu n \rrbracket \; : \; \gcd(l,n) = 1\right\}$ \\
\hline
\hline
$G_n$ & $n\ge 2$ integer & Any subgroup of $\left(\Z/\!\raisebox{-.65ex}{\ensuremath{n\Z}}\right)^{\times}$ \\
\hline
\multirow{3}{*}{$G_{n}^{(d)}$} & \multirow{3}{*}{$d\ge 1$} & Group of $d^{th}$ powers in a reduced \\
&  & system of residues modulo \\
& & a fixed integer $n\ge 2$\\
\hline 
\multirow{2}{*}{$aG_n$} & $a \in \left(\Z/\!\raisebox{-.65ex}{\ensuremath{n\Z}}\right)^{\times}$  & Coset of $a$ in the quotient of $\left(\Z/\!\raisebox{-.65ex}{\ensuremath{n\Z}}\right)^{\times}$\\
& $n\ge 2$ & by $G_n$, i.e. $aG_n = \left\{al \; : \; l\in G_n \right\} $ \\
\hline
\hline
$\Psi_X\left(aG_n\right)$ & $X>0$ & $\mathrm{Card}\left\{l\in \llbracket 1 , X \rrbracket \; : \; l\in aG_n \right\} $ \\
\hline
\multirow{2}{*}{$d_n\left(G_n\right)$} & \multirow{2}{*}{$n\ge 2$} & Index of $G_n$ in  $\left(\Z/\!\raisebox{-.65ex}{\ensuremath{n\Z}}\right)^{\times}$, i.e. \\
& & $d_n\left(G_n\right) = \varphi (n)/\Psi_n\left(G_n\right)$ \\
\hline
\hline
\end{tabular}
\caption{Some additional notation}
\label{table1}
\end{figure}

\paragraph{}
The key--step to the proof of the theorem of Duffin and Schaeffer (Theorem~\ref{dufsch}) is the study of the regularity of the distribution of the numbers less than a given positive integer and relatively prime to this integer. The following is well--known and strengthens their result in~\cite{dufsch} (Lemma~III)~: the proof, which is part of the folklore, is only given to introduce the main idea of the proof of Corollary~\ref{thminteret}.

\begin{lem}\label{repartitionprem}
Let $\mu$ be a positive real number and let $n\ge 2$ be an integer. Let $\varphi_{\mu}(n)$ denote the number of positive integers which are equal to or less than $\mu n$ and relatively prime to $n$. 

Then for any $\epsilon >0$, $$\varphi_{\mu}(n) = \varphi (n) \left(\mu + o\left(\frac{1}{n^{1-\epsilon}}\right) \right).$$ 
\end{lem}

\begin{proof}
Let $n= \prod_{i=1}^r \pi_{i}^{\alpha_i}$ be the prime factor decomposition of the integer $n\ge 2$. The sieve method provides an exact formula for $\varphi_{\mu}(n)$~: $$\varphi_{\mu}(n) \, = \,\left\lfloor \mu n\right\rfloor -\sum_{1 \le i \le r} \left\lfloor \frac{\mu n}{\pi_i} \right\rfloor + \sum_{\underset{i \neq j}{1 \le i,j \le r}} \left\lfloor \frac{\mu n}{\pi_i \pi_j} \right\rfloor  - \dots$$
Removing the integer part symbols from the above, it is not difficult to see that $$\varphi_{\mu}(n) = \mu\varphi (n) + R,$$ where the remainder $R$ satisfies $|R| \le \tau (n)$. 

Lemmas~\ref{tau} along with the inequality $\varphi (n) \ge n/2^{\omega(n)}$ valid for all positive integers imply that $$\frac{\tau(n)}{\varphi (n)} \le \frac{o\left(n^{\epsilon}\right)}{n}2^{\omega(n)} = \frac{o\left(n^{2\epsilon}\right)}{n} = o\left(\frac{1}{n^{1-2\epsilon}} \right)$$ for any $\epsilon >0$.
\end{proof}

Duffin and Schaeffer provide an error term of the form $O \left(n^{-1/2} \right)$ in Lemma~\ref{repartitionprem}, where the implied constant is absolute. In fact, even such an estimate is too accurate in the sense that their method only requires the error term to tend to zero uniformly in $\mu$. This fact shall be used to prove Corollary~\ref{thminteret}. The following theorem deals with the regularity of the distribution of the elements of a given subgroup of $\left(\Z/\!\raisebox{-.65ex}{\ensuremath{n\Z}}\right)^{\times}$ (where $n\ge 2$ is an integer) and is the key--step to the generalization of the result of Duffin and Schaeffer.

\begin{thm}\label{regularitessgrpe}
Let $\mu$ be a positive real number, $n\ge 2$ be an integer and $a\in \left(\Z/\!\raisebox{-.65ex}{\ensuremath{n\Z}}\right)^{\times}$. Let $G_n$ be a subgroup of $\left(\Z/\!\raisebox{-.65ex}{\ensuremath{n\Z}}\right)^{\times}$. Denote by $\Psi_{n} (G_n)$ the cardinality of $G_n$ (which is also the cardinality of $aG_n$) and by $d_n (G_n)$ the index of $G_n$ in  $\left(\Z/\!\raisebox{-.65ex}{\ensuremath{n\Z}}\right)^{\times}$, that is, $$d_n (G_n) = \frac{\left| \left(\Z/\!\raisebox{-.65ex}{\ensuremath{n\Z}}\right)^{\times} \right|}{\left| G_n \right|} = \frac{\varphi(n)}{\Psi_n (G_n)}\cdot$$ Finally, for a real number $\mu >0$ and an integer $n\ge 1$, let $\Psi_{\mu n}\left(aG_n\right)$ denote the number of positive integers $k$ less than or equal to $\mu n$ such that $k\in aG_n$.

Then for any $\epsilon >0$, $$\Psi_{\mu n}\left(aG_n\right) = \Psi_n(G_n) \left(\mu + o\left(\frac{d_n(G_n)}{n^{1/2-\epsilon}}\right) \right).$$
\end{thm}

\begin{proof}
The proof uses the Dirichlet characters introduced in subsection~\ref{charact} and some ideas which probably date back to the works of Erd\H{o}s and Davenport~\cite{erddav} on character sums. 

Let $H_n$ be the quotient group of $\left(\Z/\!\raisebox{-.65ex}{\ensuremath{n\Z}}\right)^{\times}$ by $G_n$. Any character $\chi$ over $H_n$ may be extended to $G_n$ by composing with the canonical homomorphism from $G_n$ to $H_n$. Such a character shall still be denoted by $\chi$. Let $\hat{G}_{H_n}$ be the set of all characters over $G_n$ arising from a character over $H_n$~: it is readily seen that $\hat{G}_{H_n}$ is a subgroup of $\hat{G_n}$ of cardinality $|\hat{H}_n|$ (here the notation of subsection~\ref{charact} is kept). 

Let $\alpha\in \left(\Z/\!\raisebox{-.65ex}{\ensuremath{n\Z}}\right)^{\times}$ be the multiplicative inverse of $a\in \left(\Z/\!\raisebox{-.65ex}{\ensuremath{n\Z}}\right)^{\times}$. By Theorem~\ref{basecaract}, $|\hat{H}_n| = d_n(G_n)$ and the same theorem implies that $$\Psi_{\mu n}\left(aG_n\right) = \frac{1}{d_n(G_n)} \sum_{ k \in \llbracket 1 , \mu n \rrbracket} \sum_{\chi \in \hat{G}_{H_n}} \chi \left(\alpha k \right).$$ On inverting the order of summation, two contributions from the sum may be distinguished~:
\begin{itemize}
\item One comes from the principal character and equals $\textrm{Card}\left(\llbracket 1 , \mu n \rrbracket \cap  \left(\Z/\!\raisebox{-.65ex}{\ensuremath{n\Z}}\right)^{\times} \right)$. Now, from Lemma~\ref{repartitionprem}, $$\textrm{Card}\left(\llbracket 1 , \mu n \rrbracket \cap  \left(\Z/\!\raisebox{-.65ex}{\ensuremath{n\Z}}\right)^{\times} \right) = \varphi_{\mu}(n) = \varphi (n) \left(\mu + o\left(\frac{1}{n^{1-\epsilon}}\right) \right)$$ for any $\epsilon >0$.

\item The other comes from the $(d_n(G_n) - 1)$ non--trivial characters and, by the P\'olya--Vinogradov inequality (Theorem~\ref{povi}), each of them is bounded above in absolute value by $2\sqrt{n}\log n$.
\end{itemize}

Therefore, for any $\epsilon >0$, $$\Psi_{ \mu n}\left(aG_n\right) = \frac{\varphi (n)}{d_n(G_n)} \left( \mu +  o\left(\frac{1}{n^{1-\epsilon}}\right)\right) + \frac{d_n(G_n)-1}{d_n(G_n)} R_n(\mu),$$ where the remainder $R_n(\mu)$ satisfies $\left|R_n(\mu)\right| \le 2\sqrt{n}\log n$. Bearing in mind that $d_n(G_n) = \varphi (n) / \Psi_n (G_n)$ and that $\varphi(n) \ge n/2^{\omega(n)}$, Lemma~\ref{tau} leads to the inequality $$\left|\frac{R_n(\mu)}{\varphi(n)} \right|\le \frac{2\sqrt{n}2^{\omega (n)} \log n }{n} = o\left(\frac{1}{n^{1/2-\epsilon}} \right)$$ for any $\epsilon >0$. This concludes the proof.
\end{proof}

The next result makes the link between Theorem~\ref{regularitessgrpe} and Corollary~\ref{thminteret} giving the repartition of the $d^{th}$ powers in a reduced system of residues modulo an integer. The notation of Theorem~\ref{regularitessgrpe} is kept.

\begin{coro}\label{regularitepuissances}
Let $n\ge 2$ and $a\ge 1$ be two coprime integers. Denote by $G_{n}^{(d)}$ the group of $d^{th}$ power residues in a reduced system of residues modulo~$n$. 

Then for all $\epsilon >0$, 
$$\Psi_{\mu n}\left(aG_{n}^{(d)}\right) = \Psi_n\left(G_{n}^{(d)}\right) \left(\mu + o\left(\frac{1}{n^{1/2-\epsilon}}\right) \right),$$ where $\Psi_n\left(G_{n}^{(d)}\right) = r_d(n) = \varphi(n) / u_d(n)$ as defined in Proposition~\ref{decompte}.
\end{coro}

\begin{proof}
Keeping the notation of Theorem~\ref{regularitessgrpe}, first notice that $d_n\left(G_{n}^{(d)}\right) = u_d(n)$.  Now, since the arithmetical function $u_d(n)$ is multiplicative (see Remark~\ref{multirdq}), Proposition~\ref{decompte} and Lemma~\ref{tau} imply that $$d_n\left(G_{n}^{(d)}\right) = u_d(n) \le (2d)^{\omega(n)} = o\left(n^{\epsilon} \right)$$ for any $\epsilon >0$. The result then follows from Theorem~\ref{regularitessgrpe}.
\end{proof}

To prove Corollary~\ref{thminteret}, the following notation is convenient.

\begin{nota}
For any real number $x \in [0 , 1/2)$ and any integer $k\ge 1$, let $E_k^{x}$ denote the collection of intervals of the form $$\left(\frac{p}{q_k^d}-\frac{x}{q_k^{d}} \, , \, \frac{x}{q_k^{d}}+\frac{p}{q_k^d} \right)$$ where $0 < p < q_k^d$ is an integer relatively prime to $q_k$ and satisfying $p \equiv a b^d \imod{q_k}$ for an integer $b$ prime to $q_k$ (with the notation of Corollary~\ref{regularitepuissances}, this amounts to claiming that $p \in \llbracket 0 , q_k^d \rrbracket$ and that $p\in aG_{q_k}^{(d)}$ ).
Here and in what follows, the integer $a$ is fixed and assumed to be relatively prime to $q_k$ for all $k\ge 1$.

For simplicity, set furthermore $E_k:=E^{\alpha_k}_{k}$ for all integers $k\ge 1$.
\end{nota}

As mentioned in~\cite{sprin} (p.27), it is enough to consider the case where the sequence $\left(\alpha_k\right)_{k\ge 1}$ in Corollary~\ref{thminteret} takes its values in the interval $[0 , 1/2)$. This assumption can be dropped, but this leads to some additional complications which are not of interest.

With the notation of Corollary~\ref{regularitepuissances}, $E_k$ is the set in $(0,1)$ consisting of 
\begin{align}\label{psidqk}
\Psi_{q_k^d}\left(a G_{q_k}^{(d)}\right) = \Psi_{q_k^d}\left(G_{q_k}^{(d)}\right) = \Psi_{q_k}\left(G_{q_k}^{(d)}\right)q_k^{d-1}
\end{align} 
open intervals each of length $2\alpha_k/q_k^d$ with centers at $p/q_k^d$, where $p$ and $q_k$ are integers satisfying the aforementioned constraints ($\Psi_{q_{k}^d}\left(a G_{q_k}^{(d)}\right)$ is the number of integers $p \in \llbracket 0 , q_k^d \rrbracket$ such that $p\in aG_{q_k}^{(d)}$. From the fact that the integer $a$ is coprime with $q_k$, it should be obvious that $\Psi_{q_{k}^d}\left(a G_{q_k}^{(d)}\right) = \Psi_{q_{k}^d}\left(G_{q_k}^{(d)}\right)$).

If $(s, t)$ is some interval in $(0,1)$, an estimate of the measure of the set common to $E_k$ and the interval $(s, t)$ is needed. To that end, notice that, for any integer $n\ge 1$ and any real number $\mu >0$, $\Psi_{\mu n^d}\left(aG_{q_k}^{(d)}\right)$  counts the number of positive integers $p$ less than or equal to $\mu n^d$ such that $p\in aG_{q_k}^{(d)}$. 

Let $k\ge 1$ be an integer. The number of intervals in $E_k$ whose centers lie in $(s, t)$ is exactly $\Psi_{tq_k^d}\left(aG_{q_k}^{(d)}\right) - \Psi_{sq_k^d}\left(aG_{q_k}^{(d)}\right)$. From this it follows that at least $\Psi_{tq_k^d}\left(aG_{q_k}^{(d)}\right) - \Psi_{sq_k^d}\left(aG_{q_k}^{(d)}\right) - 2$ such intervals are entirely contained in $(s, t)$ and at most $\Psi_{tq_k^d}\left(aG_{q_k}^{(d)}\right) - \Psi_{sq_k^d}\left(aG_{q_k}^{(d)}\right) +2$ of them touch $(s,t)$. Thus the measure of the set common to $E_k$ and $(s, t)$ is 
\begin{align}\label{chevauchementE_k}
\frac{2\alpha_k}{q_k^d}\left(\Psi_{tq_k^d}\left( aG_{q_k}^{(d)}\right) - \Psi_{sq_k^d}\left( aG_{q_k}^{(d)}\right) + \theta \right),
\end{align}
where $\left| \theta\right| \le 2$.

However, since for any $\mu >0$, $\left\lfloor \mu q_k^{d-1} \right\rfloor$ is the greatest integer $m$ satisfying $mq_k \le \mu q_k^d$, we get $$\Psi_{\mu q_k^d}\left(aG_{q_k}^{(d)}\right) = \left\lfloor \mu q_k^{d-1} \right\rfloor \Psi_{q_k}\left(G_{q_k}^{(d)}\right) + \textrm{Card} \left\{p \in \llbracket \left\lfloor \mu q_k^{d-1} \right\rfloor q_k \, , \, \mu q_k^d \rrbracket\; :\; p\in aG_{q_k}^{(d)}\right\}.$$ The second term on the right-hand side of this equation is $\Psi_{\nu q_k}\left(aG_{q_k}^{(d)}\right),$ where $$\nu := \frac{\mu q^{d}_k - \left\lfloor \mu q_k^{d-1} \right\rfloor q_k}{q_k} \, \in \, [0, 1).$$ Therefore, from Corollary~\ref{regularitepuissances},
\begin{align*}
\Psi_{\mu q_k^d}\left(aG_{q_k}^{(d)}\right) &=  \left\lfloor \mu q_k^{d-1} \right\rfloor \Psi_{q_k}\left(G_{q_k}^{(d)}\right) + \Psi_{\nu q_k}\left(aG_{q_k}^{(d)}\right) \\
&=  \Psi_{q_k}\left(G_{q_k}^{(d)}\right) \left(\left\lfloor \mu q_k^{d-1} \right\rfloor + \mu q_k^{d-1} - \left\lfloor \mu q_k^{d-1} \right\rfloor + o\left(\frac{1}{q_k^{1/2-\epsilon}} \right) \right) \\
&= \Psi_{q_k^d}\left(G_{q_k}^{(d)}\right)  \left( \mu  + o\left(\frac{1}{q_k^{d-1/2-\epsilon}} \right) \right)
\end{align*}
for any $\epsilon >0$.

Putting this into~(\ref{chevauchementE_k}) and denoting by $\lambda$ the one--dimensional Lebesgue measure, the measure of the set common to $E_k$ and $(s,t)$ is seen to be $$\frac{2\alpha_k}{q_k^d}\Psi_{q_k^d}\left(G_{q_k}^{(d)}\right) \left(t-s + o(1) \right) \, =\, \lambda\left(E_k \right) (t-s) (1+o(1)),$$ where the last $o(1)$ is less than $\left(q_k^{d-1/2-\epsilon}(t-s) \right)^{-1}$ for any $\epsilon >0$.

Thus the following lemma has almost been proven.

\begin{lem}\label{lem58}
Let $A$ be a subset of the unit interval $(0,1)$ consisting of a finite number of intervals.

Then, there exists a constant $c_A >0$ which depends only on the set $A$ such that for any integer $k\ge 1$, $$\lambda\left(A \cap E_k \right) \le \lambda\left(A \right) \lambda\left(E_k \right)\left(1+c_A \,\rho\left(q_k \right) \right),$$ where $$\rho\left(q_k \right)  = o\left(\frac{1}{q_k^{d-1/2-\epsilon}} \right)$$ for any $\epsilon >0.$
\end{lem}

\begin{proof}
The lemma has been proven in the case where $A$ is a single interval. The general case follows easily. See Lemma~IV in~\cite{dufsch}.
\end{proof}

All the tools necessary for the proof of Corollary~\ref{thminteret} are now available. In fact, the proof has been reduced to that of the theorem of Duffin and Schaeffer, which may be found in~\cite{dufsch} (p. 248 to 250). In the latter, the reference to Lemma~IV should be replaced by the reference to Lemma~\ref{lem58} in the above and inequalities~(13) should be read as follows~:
\begin{quotation}
\emph{By assumption, there are arbitrarily large integers $n$ and $m$ such that $$\sum_{j=n}^{m}\alpha_j > 1 \; \textrm{ and }\; \sum_{j=n}^{m}\alpha_j s_d\left(q_k^d\right) > \frac{1}{2}c\sum_{j=n}^{m}\alpha_j, $$ where $$s_d\left(q_k^d\right) = \frac{\Psi_{q_k^d}\left(G_{k}^{(d)}\right)}{q_k^d}\cdotp$$}
\end{quotation}
For the latter, see the definitions of $s_d(q)$ in~(\ref{defsdq}), of $\Psi_{q_k}\left(G_{k}^{(d)}\right)$ in Corollary~\ref{regularitepuissances} and of $\Psi_{q_k^d}\left(G_{k}^{(d)}\right)$ in~(\ref{psidqk}).

This concludes the proof of Corollary~\ref{thminteret}.

\subsection{The proof of Theorem~\ref{extensiondf}}

In the course of the proof of Corollary~\ref{thminteret}, the main step was the proof of Theorem~\ref{regularitessgrpe} and the fact that the subgroup $G_{n}^{(d)}$ of $\left(\Z/\!\raisebox{-.65ex}{\ensuremath{n\Z}}\right)^{\times}$ was sufficiently large in the sense that, for some $\epsilon >0$, $$\frac{d_n\left(G_n^{(d)}\right)}{n^{1/2-\epsilon}}\longrightarrow 0$$ as $n$ tends to infinity, with the notation of Corollary~\ref{regularitepuissances}. Otherwise, no use whatsoever of any specific property of the group of $d^{th}$ powers in a reduced system of residues modulo $n$ was made. Consequently, apart from some minor modifications, the same proof as that provided for Corollary~\ref{thminteret} demonstrates Theorem~\ref{extensiondf}.

Note that in Corollary~\ref{thminteret}, the denominators of the rational approximants were prescribed to be $d^{th}$~powers. 

\begin{rem}
Condition~(c) in Theorem~\ref{extensiondf} is derived from the fact that the P\'olya--Vinogradov inequality (Theorem~\ref{povi}) gives $2\sqrt{n}\log n$ as an upper bound for the absolute value of the sum of values of a non--principal Dirichlet character to the modulus~$n$ and the fact that $$ \frac{2\sqrt{n}2^{\omega (n)} \log n }{n} = o\left(\frac{1}{n^{1/2-\epsilon}} \right)$$ for any $\epsilon >0$ (see the proof of Theorem~\ref{regularitessgrpe}). Therefore, any improvement of the P\'olya--Vinogradov inequality would lead to a condition weaker than~(c). However, stated in this form, the exponent $1/2 - \epsilon$ for some $\epsilon >0$ appearing in condition~(c) cannot be improved if a general result is required~: indeed, assuming the Riemann Hypothesis for $L$--functions (i.e. the Generalized Riemann Hypothesis),  E.Bach~\cite{ebach} has shown that a sharper upper bound for the sum of values of a non--principal Dirichlet character to the modulus~$n$ was $2\sqrt{n}\log\log n$. Up to a constant, this is best possible since in 1932 Paley~\cite{paley} proved that there exist infinitely many quadratic characters~$\chi$ (i.e. characters of the form $\chi(n) = \left(\frac{n}{m} \right)$ for some odd integer $m$, where $\left(\frac{n}{m} \right)$ is the Jacobi symbol) such that there exists a constant~$c>0$ which satisfy for some $N\in\N^*$ the inequality $$\left|\sum_{n=1}^{N}\chi (n) \right|> c  \sqrt{n} \log \log n .$$
\end{rem}

\renewcommand{\abstractname}{Acknowledgements}
\begin{abstract}
The author would like to thank his PhD supervisor Detta Dickinson for suggesting the problem and for discussions which helped to develop ideas put forward. He is supported by the Science Foundation Ireland grant RFP11/MTH3084.
\end{abstract}

\bibliographystyle{plain}
\bibliography{extension}

\end{document}